\documentclass{amsart}
\usepackage{graphicx}
\usepackage{amsmath}

 \newtheorem{theorem}{Theorem}[section]
 \newtheorem{corollary}[theorem]{Corollary}
 \newtheorem{lemma}[theorem]{Lemma}
 \newtheorem{remark}[theorem]{Remark}

\begin{document}
\title{Structure of semisimple
Hopf algebras of dimension $p^2q^2$, II}

\author{Jingcheng Dong}
\address{College of Engineering, Nanjing Agricultural University, Nanjing
210031, Jiangsu, People's Republic of China}
\email{dongjc@njau.edu.cn}

 \maketitle
\begin{abstract}
Let $k$ be an algebraically closed field of characteristic $0$. In
this paper, we obtain the structure theorems for semisimple Hopf
algebras of dimension $p^2q^2$ over $k$, where $p,q$ are prime
numbers with $p^2<q$. As an application, we also obtain the
structure theorems for semisimple Hopf algebras of dimension $9p^2$
and $25q^2$ for all primes $3\neq p$ and $5\neq q$.
\end{abstract}



\section{Introduction}\label{sec1}
Throughout this paper, we will work over an algebraically closed
field $k$ of characteristic $0$.

Quite recently, an outstanding classification result was obtained
for semisimple Hopf algebras over $k$.  That is, Etingof et al
\cite{Etingof2} completed the classification of semisimple Hopf
algebras of dimension $pq^2$ and $pqr$, where $p,q,r$ are distinct
prime numbers. The results in \cite{Etingof2} showed that all these
Hopf algebras can be constructed from group algebras and their duals
by means of extensions. Up to now, besides those mentioned above,
semisimple Hopf algebras of dimension $p,p^2,p^3$ and $pq$ have been
completely classified. See
\cite{Etingof,Gelaki,Masuoka,Masuoka2,Masuoka3,Zhu} for details.

Recall that a semisimple Hopf algebra $H$ is called of Frobenius
type if the dimensions of the simple $H$-modules divide the
dimension of $H$. Kaplansky conjectured that every
finite-dimensional semisimple Hopf algebra is of Frobenius type
\cite[Appendix 2]{Kaplansky}. It is still an open problem. Many
examples show that a positive answer to Kaplansky's conjecture would
be very helpful in the classification of semisimple Hopf algebras.
See \cite{dong1} and the examples mentioned above for details.

In a previous paper \cite{dong}, we studied the structure of
semisimple Hopf algebras of dimension $p^2q^2$, where $p,q$ are
prime numbers with $p^4<q$. As an application, we also studied the
structure of semisimple Hopf algebras of dimension $4q^2$, where $q$
is a prime number. In the present paper, we shall continue our
investigation and prove that the main results in \cite{dong} can be
extended to the case $p^2<q$. Moreover, the structure theorems for
semisimple Hopf algebras of dimension $9p^2$ and $25q^2$ will also
be given in this paper, where $3\neq p$ and $5\neq q$ are prime
numbers.

The paper is organized as follows. In Section \ref{sec2}, we recall
the definitions and basic properties of semisolvability, characters
and Radford's biproducts, respectively. Some useful lemmas are also
obtained in this section. In particular, we give an partial answer
to Kaplansky's conjecture. We prove that if ${\rm dim}H$ is odd and
$H$ has a simple module of dimension $3$ then $3$ divides ${\rm
dim}H$. Under the assumption that $H$ does not have simple modules
of dimension $3$ and $7$, we also prove that if ${\rm dim}H$ is odd
and $H$ has a simple module of dimension $5$ then $5$ divides ${\rm
dim}H$.

We begin our main work in Section \ref{sec3}. Let $H$ be a
semisimple Hopf algebras of dimension $p^2q^2$, where $p<q$ is a
prime number. We first prove that if $|G(H^*)|=q^2$ then $H$ is
upper semisolvable, in the sense of \cite{Montgomery}. It is a
generalization of \cite[Lemma 3.4]{dong}. We then present our main
result. We prove that if $p^2<q$ then $H$ is either semisolvable or
isomorphic to a Radford's biproduct $R\# kG$, where $kG$ is the
group algebra of group $G$ of order $p^2$, $R$ is a semisimple
Yetter-Drinfeld Hopf algebra in ${}^{kG}_{kG}\mathcal{YD}$ of
dimension $q^2$. Our approach is mainly based on looking for normal
Hopf subalgebras of $H$ of dimension $pq^2$. In Section \ref{sec4},
we shall study the structure of semisimple Hopf algebras of
dimension $9p^2$ and $25q^2$.

Throughout this paper, all modules and comodules are left modules
and left comodules, and moreover they are finite-dimensional over
$k$. $\otimes$, ${\rm dim}$ mean $\otimes _k$, ${\rm dim}_k$,
respectively. Our references for the theory of Hopf algebras are
\cite{Montgomery2} or \cite{Sweedler}. The notation for Hopf
algebras is standard. For example, the group of group-like elements
in $H$ is denoted by $G(H)$.

\section{Preliminaries}\label{sec2}
\subsection{Characters}\label{sec2-1}
Throughout this subsection, $H$ will be a semisimple Hopf algebra
over $k$. As an algebra,  $H$ is isomorphic to a direct product of
full matrix algebras $$H\cong k^{(n_1)}\times
\prod_{i=2}^{s}M_{d_i}(k)^{(n_i)},$$ where $n_1=|G(H^*)|$. In this
case, we say $H$ is of type $(d_1,n_1;\cdots;d_s,n_s)$ as an
algebra, where $d_1=1$. If $H^*$ is of type
$(d_1,n_1;\cdots;d_s,n_s)$ as an algebra, we shall say that $H$ is
of type $(d_1,n_1;\cdots;d_s,n_s)$ as a coalgebra.

Obviously, $H$ is of type $(d_1,n_1;\cdots;d_s,n_s)$ as an algebra
if and only if $H$ has $n_1$ non-isomorphic irreducible characters
of degree $d_1$, $n_2$ non-isomorphic irreducible characters of
degree $d_2$, etc. In this paper, we shall use the notation $X_t$ to
denote the set of all irreducible characters of $H$ of degree $t$.

Let $V$ be an $H$-module. The character of $V$ is the element
$\chi=\chi_V\in H^*$ defined by $\langle\chi,h\rangle={\rm Tr}_V(h)$
for all $h\in H$. The degree of $\chi$ is defined to be the integer
${\rm deg}\chi=\chi(1)={\rm dim}V$. If $U$ is another $H$-module, we
have $$\chi_{U\otimes V}=\chi_U\chi_V,\quad\chi_{V^*}=S(\chi_V),$$
where $S$ is the antipode of $H^*$.

All irreducible characters of $H$ span a subalgebra $R(H)$ of $H^*$,
which is called the character algebra of $H$. By \cite[Lemma
2]{Zhu}, $R(H)$ is semisimple. The antipode $S$ induces an
anti-algebra involution $*: R(H)\to R(H)$, given by
$\chi\mapsto\chi^*:=S(\chi)$. The character of the trivial
$H$-module is the counit $\varepsilon$.

Let $\chi_U,\chi_V\in R(H)$ be the characters of the $H$-modules $U$
and $V$, respectively. The integer $m(\chi_U,\chi_V)={\rm
dimHom}_H(U,V)$ is defined to be the multiplicity of $U$ in $V$.
This can be extended to a bilinear form $m:R(H)\times R(H)\to k$.

Let ${\rm Irr}(H)$ denote the set of irreducible characters of $H$.
Then ${\rm Irr}(H)$ is a basis of $R(H)$. If $\chi\in R(H)$, we may
write $\chi=\sum_{\alpha\in {\rm Irr}(H)}m(\alpha,\chi)\alpha$. Let
$\chi,\psi,\omega\in R(H)$. Then
$m(\chi,\psi\omega)=m(\psi^*,\omega\chi^*)=m(\psi,\chi\omega^*)$ and
$m(\chi,\psi)=m(\chi^*,\psi^*)$. See \cite[Theorem 9]{Nichols}.

For each group-like element $g$ in  $G(H^*)$, we have
$m(g,\chi\psi)=1$, if $\psi=\chi^*g$ and $0$ otherwise for all
$\chi,\psi\in {\rm Irr}(H)$. In particular, $m(g,\chi\psi)=0$ if
${\rm deg}\chi\neq {\rm deg}\psi$. Let $\chi\in {\rm Irr}(H)$. Then
for any group-like element $g$ in $G(H^*)$, $m(g,\chi\chi^{*})>0$ if
and only if $m(g,\chi\chi^{*})= 1$ if and only if $g\chi=\chi$. The
set of such group-like elements forms a subgroup of $G(H^*)$, of
order at most $({\rm deg}\chi)^2$.  See \cite[Theorem 10]{Nichols}.
Denote this subgroup by $G[\chi]$. In particular, we have
$$\chi\chi^*=\sum_{g\in G[\chi]}g+\sum_{\alpha\in {\rm Irr}(H),{\rm
deg}\alpha>1}m(\alpha,\chi\chi^*)\alpha.$$

A subalgebra $A$ of $R(H)$ is called a standard subalgebra if $A$ is
spanned by irreducible characters of $H$. Let $X$ be a subset of
${\rm Irr}(H)$. Then $X$ spans a standard subalgebra of $R(H)$ if
and only if the product of characters in $X$ decomposes as a sum of
characters in $X$. There is a bijection between $*$-invariant
standard subalgebras of $R(H)$ and quotient Hopf algebras of $H$.
See \cite[Theorem 6]{Nichols}.

In the rest of this subsection, we shall present some results on
irreducible characters and algebra types.

\begin{lemma}\label{lem1}
Let $\chi\in {\rm Irr}(H)$ be an irreducible character of $H$. Then

 (1)\,The order of $G[\chi]$ divides $({\rm deg}\chi)^2$.

 (2)\,The order of $G(H^*)$ divides $n({\rm deg}\chi)^2$, where $n$ is the
 number of non-isomorphic irreducible characters of degree ${\rm deg}\chi$.
\end{lemma}
\begin{proof} It follows from Nichols-Zoeller Theorem \cite{Nichols2}. See
also \cite[Lemma 2.2.2]{Natale1}.
\end{proof}

\begin{lemma}\label{lem2}
Assume  that ${\rm dim}H$ is odd and $H$ is of type
$(1,n_1;\cdots;d_s,n_s)$ as an algebra. Then $d_i$ is odd and $n_i$
is even for all $2\leq i\leq s$.
\end{lemma}
\begin{proof} It follows from \cite[Theorem 5]{Kashina} that $d_i$ is odd.

If there exists $i\in \{2,\cdots,s\}$ such that $n_i$ is odd, then
there is at least one irreducible character of degree $d_i$ such
that it is self-dual. This contradicts \cite[Theorem 4]{Kashina}.
\end{proof}

\begin{lemma}\label{lem3}
Assume that ${\rm dim}H$ is odd. If $H$ has a simple module of
dimension $3$, then $3$ divides the order of $G(H^*)$. In
particular, $3$ divides ${\rm dim}H$.
\end{lemma}
\begin{proof} Let $\chi_3$ be an irreducible character of degree $3$. By
Lemma \ref{lem2}, $H$ does not have irreducible characters of even
degree. Therefore, if $G[\chi_3]$ is trivial then
$\chi_3\chi_3^*=\varepsilon+\chi_3'+\chi_5$ for some $\chi_3'\in
X_3,\chi_5\in X_5$. Since $\chi_3\chi_3^*$ is self-dual, $\chi_3'$
and $\chi_5$ are self-dual. It contradicts the assumption and
\cite[Theorem 4]{Kashina}. Hence, $G[\chi]$ is not trivial for every
$\chi\in X_3$. By Lemma \ref{lem1} (1), the order of $G[\chi]$ is
$3$ or $9$. Thus, $3$ divides $|G(H^*)|$ since $G[\chi]$ is a
subgroup of $G(H^*)$ for every $\chi\in X_3$.

The second statement can be obtained by the Nichols--Zoeller
Theorem.
\end{proof}
\begin{remark}
The above lemma has appeared in \cite[Corollary]{Burciu} and
\cite[Theorem 4.4]{Kashina2}, respectively. In the first paper,
Burciu does not assume that the characteristic of the base field is
zero, but adds the assumption that $H$ has no even-dimensional
simple modules. Accordingly, his proof is rather different from
ours. The author learned the result in the second paper after he
finished this paper. Our proof here is slightly different from that
in the second paper. So we give the proof for the sake of
completeness.
\end{remark}

\begin{corollary}\label{lem4}
Assume  that ${\rm dim}H$ is odd and $H$ is of type
$(1,n;3,m;\cdots)$ as an algebra. If

(1)\,$H$ does not have irreducible characters of degree $9$, or

(2)\,there exists a non-trivial subgroup $G$ of $G(H^*)$ such that
$G[\chi]=G$ for all $\chi\in X_3$,

then $H$ has a quotient Hopf algebra of dimension $n+9m$.
\end{corollary}
\begin{proof} Let $\chi,\psi$ be irreducible characters of degree $3$. By
assumption and \cite[Lemma 2.5]{dong}, $\chi\psi$ is not
irreducible. If there exists $\chi_5\in X_5$ such that
$m(\chi_5,\chi\psi)>0$ then $\chi\psi=\chi_5+\chi_3+g$ for some
$\chi_3\in X_3$ and $g\in G(H^*)$, by Lemma \ref{lem2}. From
$m(g,\chi\psi)=m(\chi,g\psi^*)=1$, we get $\chi=g\psi^*$. Then
$\chi\psi=g\psi^*\psi=\chi_5+\chi_3+g$ shows that
$\psi^*\psi=g^{-1}\chi_5+g^{-1}\chi_3+\varepsilon$. This contradicts
Lemma \ref{lem3}. Similarly, we can show that there does not exist
$\chi_7\in X_7$ such that $m(\chi_7,\chi\psi)>0$. Therefore,
$\chi\psi$ is a sum of irreducible characters of degree $1$ or $3$.
It follows that irreducible characters of degree $1$ and $3$ span a
standard subalgebra of $R(H)$ and $H$ has a quotient Hopf algebra of
dimension $n+9m$.
\end{proof}
\begin{lemma}\label{lem4-1}
Assume that ${\rm dim}H$ is odd and $H$ does not have simple modules
of dimension $3$ and $7$. If $H$ has a simple module of dimension
$5$, then $5$ divides the order of $G(H^*)$. In particular, $5$
divides ${\rm dim}H$.
\end{lemma}
\begin{proof} Let $\chi$ be an irreducible character of degree $5$. By
assumption and Lemma \ref{lem2}, if $G[\chi]$ is trivial then there
are four possible decomposition of $\chi\chi^*$:
$$\chi\chi^*=\varepsilon+\chi_{11}+\chi_{13};\chi\chi^*=\varepsilon+\chi_{9}+\chi_{15};
\chi\chi^*=\varepsilon+\chi_{5}+\chi_{19};
\chi\chi^*=\varepsilon+\chi_{5}^1+\chi_{5}^2+\chi_{5}^3+\chi_{9},$$
where $\chi_i,\chi_j^k$ are irreducible characters of degree $i,j$.
In all cases, there exists at least one irreducible character such
that it is self-dual, since $\chi\chi^*$ is self-dual. It
contradicts the assumption and \cite[Theorem 4]{Kashina}. Therefore,
$G[\chi]$ is not trivial  for every $\chi\in X_5$. Hence, $5$
divides the order of $G(H^*)$ by Lemma \ref{lem1} (1).
\end{proof}

\subsection{Semisolvability}\label{sec2-2}
Let $B$ be a finite-dimensional Hopf algebra over $k$. A Hopf
subalgebra $A\subseteq B$ is called normal if $h_1AS(h_2)\subseteq
A$ and $S(h_1)Ah_2\subseteq A$, for all $h\in B$. If $B$ does not
contain proper normal Hopf subalgebras then it is called simple. The
notion of simplicity is self-dual, that is, $B$ is simple if and
only if $B^*$ is simple.

The notions of upper and lower semisolvability for
finite-dimensional Hopf algebras have been introduced in
\cite{Montgomery}, as generalizations of the notion of solvability
for finite groups. By definition, $H$ is called lower semisolvable
if there exists a chain of Hopf subalgebras
$$H_{n+1} = k\subseteq
H_{n}\subseteq\cdots \subseteq H_1 = H$$ such that $H_{i+1}$ is a
normal Hopf subalgebra of $H_i$, for all $i$, and all quotients
$H_{i}/H_{i}H^+_{i+1}$ are trivial. That is, they are isomorphic to
a group algebra or a dual group algebra. Dually, $H$ is called upper
semisolvable if there exists a chain of quotient Hopf algebras
$$H_{(0)} =
H\xrightarrow{\pi_1}H_{(1)}\xrightarrow{\pi_2}\cdots\xrightarrow{\pi_n}H(n)
= k$$ such that $H_{(i-1)}^{co\pi_{i}}=\{h\in H_{(i-1)}|(id\otimes
\pi_i)\Delta(h)=h\otimes 1\}$ is a normal Hopf subalgebra of
$H_{(i-1)}$, and all $H_{(i-1)}^{co\pi_i}$ are trivial.

In analogy with the situations for finite groups, it is enough for
many applications to know that a Hopf algebra is semisolvable.

By \cite[Corollary 3.3]{Montgomery}, we have that $H$ is upper
semisolvable if and only if $H^*$ is lower semisolvable. If this is
the case, then $H$ can be obtained from group algebras and their
duals by means of (a finite number of) extensions.

\subsection{Radford's biproduct}\label{sec2-3}Let $A$ be a
semisimple Hopf algebra and let ${}^A_A\mathcal{YD}$ denote the
braided category of Yetter-Drinfeld modules over $A$. Let $R$ be a
semisimple Yetter-Drinfeld Hopf algebra in ${}^A_A\mathcal{YD}$.
Denote by $\rho :R\to A\otimes R$, $\rho (a)=a_{-1} \otimes a_0 $,
and $\cdot :A\otimes R\to R$, the coaction and action of $A$ on $R$,
respectively. We shall use the notation $\Delta (a)=a^1\otimes a^2$
and $S_R $ for the comultiplication and the antipode of $R$,
respectively.

Since $R$ is in particular a module algebra over $A$, we can form
the smash product (see \cite[Definition 4.1.3]{Montgomery}). This is
an algebra with underlying vector space $R\otimes A$, multiplication
is given by $$(a\otimes g)(b\otimes h)=a(g_1 \cdot b)\otimes g_2 h,
\mbox{\;for all\;}g,h\in A,a,b\in R,$$ and unit $1=1_R\otimes1_A$.

Since $R$ is also a comodule coalgebra over $A$, we can dually form
the smash coproduct. This is a coalgebra with underlying vector
space $R\otimes A$, comultiplication is given by $$\Delta (a\otimes
g)=a^1\otimes (a^2)_{-1} g_1 \otimes (a^2)_0 \otimes g_2
,\mbox{\;for all\;}h\in A,a\in R, $$ and counit
$\varepsilon_R\otimes\varepsilon_A$.

As observed by D. E. Radford (see \cite[Theorem 1]{Radford}), the
Yetter-Drinfeld condition assures that $R\otimes A$ becomes a Hopf
algebra with these structures. This Hopf algebra is called the
Radford's biproduct of $R$ and $A$. We denote this Hopf algebra by $
R\#A$ and write $a\# g=a\otimes g$ for all $g\in A,a\in R$. Its
antipode is given by
$$S(a\# g)=(1\# S(a_{-1} g))(S_R (a_0 )\# 1),\mbox{\;for
all\;}g\in A,a\in R.$$

A biproduct $R\#A$ as described above is characterized by the
following property(see \cite[Theorem 3]{Radford}): suppose that $H$
is a finite-dimensional Hopf algebra endowed with Hopf algebra maps
$\iota:A\to H$ and $\pi:H\to A$ such that $\pi \iota:A\to A$ is an
isomorphism. Then the subalgebra $R= H^{co\pi}$ has a natural
structure of Yetter-Drinfeld Hopf algebra over $A$ such that the
multiplication map $R\#A\to H$ induces an isomorphism of Hopf
algebras.

The following lemma is a special case of \cite[Lemma
4.1.9]{Natale4}.
\begin{lemma}\label{lem5}
Let $H$ be a semisimple Hopf algebra of dimension $p^2q^2$, where
$p,q$ are distinct prime numbers. If $gcd(|G(H)|,|G(H^*)|)=p^2$,
then $H\cong R\#kG$ is a biproduct, where $kG$ is the group algebra
of group $G$ of order $p^2$, $R$ is a semisimple Yetter-Drinfeld
Hopf algebra in $^{kG}_{kG}\mathcal{YD}$ of dimension $q^2$.
\end{lemma}

\section{Semisimple Hopf algebras of dimension $p^2q^2$}\label{sec3}
Let $p,q$ be distinct prime numbers with $p<q$. Throughout this
section, $H$ will be a semisimple Hopf algebra of dimension
$p^2q^2$, unless otherwise stated. By Nichols-Zoeller Theorem
\cite{Nichols2}, the order of $G(H^*)$ divides ${\rm dim}H$.
Moreover, $|G(H^*)|\neq1$ by \cite[Proposition 9.9]{Etingof2}. By
\cite[Lemma 1]{dong}, $H$ is of Frobenius type. Therefore, the
dimension of a simple $H$-module can only be $1,p,p^2$ or $q$. Let
$a,b,c$ be the number of non-isomorphic simple $H$-modules of
dimension $p,p^2$ and $q$, respectively. It follows that we have an
equation $p^2q^2=|G(H^*)|+ap^2+bp^4+cq^2$. In particular, if
$|G(H^*)|=p^2q^2$ then $H$ is a dual group algebra; if
$|G(H^*)|=pq^2$ then $H$ is upper semisolvable by the following
lemma, which  is due to \cite[Lemma 2.3]{dong}.

\begin{lemma}\label{lem6}
If $H$ has a Hopf subalgebra $K$ of dimension $pq^2$ then $H$ is
lower semisolvable.
\end{lemma}
The following lemma is a refinement of  \cite[Lemma 3.4]{dong}.
\begin{lemma}\label{lem7}
If the order of $G(H^*)$ is $q^2$ then $H$ is upper semisolvable.
\end{lemma}
\begin{proof} If $p=2$ and $q=3$ then it is the case discussed in
\cite[Chapter 8]{Natale4}. Hence, $H$ is upper semisolvable.
Throughout the remainder of the proof, we assume that $p\geq 3$.

By Lemma \ref{lem1} (2), if $a\neq 0$ then $ap^2\geq p^2q^2$, a
contradiction. Hence, $a=0$. Similarly, $b=0$. If follows that $H$
is of type $(1,q^2;q,p^2-1)$ as an algebra.

The group $G(H^*)$ acts by left multiplication on the set $X_q$. The
set $X_q$ is a union of orbits which have length $1,q$ or $q^2$.
Since $q>p\geq 3$, $q$ does not divides $p^2-1$. Therefore, there
exists one orbit with length $1$. That is, there exists an
irreducible character $\chi_q\in X_q$ such that $G[\chi_q]=G(H^*)$.
This means that $g\chi_q=\chi_q=\chi_qg$ for all $g\in G(H^*)$.

Let $C$ be a $q^2$-dimensional simple subcoalgebra of $H^*$,
corresponding to $\chi_q$. Then $gC=C=Cg$ for all $g\in G(H^*)$. By
\cite[Proposition 3.2.6]{Natale4}, $G(H^*)$ is normal in $k[C]$,
where $k[C]$ denotes the subalgebra generated by $C$. It is a Hopf
subalgebra of $H^*$ containing $G(H^*)$.  Counting dimension, we
know ${\rm dim}k[C]\geq 2q^2$. Since ${\rm dim}k[C]$ divides ${\rm
dim}H$, we know ${\rm dim}k[C]=pq^2$ or $p^2q^2$. If ${\rm
dim}k[C]=pq^2$ then Lemma \ref{lem6} shows that $H^*$ is lower
semisolvable. If  ${\rm dim}k[C]=p^2q^2$ then $k[C]=H^*$. Since
$kG(H^*)$ is a group algebra and the quotient $H^*/H^*(kG(H^*))^+$
is trivial (see \cite{Masuoka2}), $H^*$ is lower semisolvable.
Hence, $H$ is upper semisolvable. This completes the proof.
\end{proof}

\begin{theorem}\label{thm1}
If $q>p^2$ then $H$ is either semisolvable or isomorphic to a
Radford's biproduct $R\# kG$, where $kG$ is the group algebra of
group $G$ of order $p^2$, $R$ is a semisimple Yetter-Drinfeld Hopf
algebra in ${}^{kG}_{kG}\mathcal{YD}$ of dimension $q^2$.
\end{theorem}
\begin{proof} By \cite[Proposition 1.1]{Artamonov}, $H$ has a quotient Hopf
algebra $\overline{H}$ of dimension $|G(H^*)|+ap^2+bp^4$. In
particular, $|G(H^*)|$ divides ${\rm dim}\overline{H}$ and $|
G(H^*)|+ap^2+bp^4$ divides ${\rm dim}H$.

We first prove that the order of $G(H^*)$ can not be $q$. Suppose on
the contrary that $|G(H^*)|= q$. We first note that $c\neq0$, since
otherwise we get the contradiction $p^2\mid q$. Since $q$ divides
${\rm dim}\overline{H}$ and $c\neq 0$, we have that ${\rm
dim}\overline{H}<p^2q^2$. Therefore, ${\rm
dim}\overline{H}=q,pq,p^2q,pq^2$ or $q^2$. If ${\rm
dim}\overline{H}=q^2$ then $(\overline{H})^*\subseteq kG(H^*)$ by
\cite{Masuoka2}. It is impossible since $q^2 ={\rm dim}\overline{H}$
does not divide $|G(H^*)|= q$. If ${\rm dim}\overline{H}=q,pq$ or
$p^2q$ then we have $p^2q^2= q+cq^2$, $p^2q^2=pq+cq^2$ or
$p^2q^2=p^2q+cq^2$. They all impossible. Hence, ${\rm
dim}\overline{H}=p^2q$. That is $q+ap^2+bp^4=pq^2$. It is
impossible, too.

We then prove that if $|G(H^*)|=p$ or $pq$ then $H$ is upper
semisolvable. We first note that $c\neq0$, since otherwise we get
the contradiction $p^2\mid p$. Then $p\mid{\rm dim}\overline{H}$ and
${\rm dim}\overline{H} < p^2q^2$. Therefore ${\rm dim}\overline{H} =
p, pq, p^2q, pq^2$ or $p^2$. Moreover, ${\rm dim}\overline{H}\neq
p^2$, since otherwise $(\overline{H})^*\subseteq kG(H^*)$ by
\cite{Masuoka2}, but $p^2 ={\rm dim}\overline{H}$ does not divide
$|G(H^*)|= p$ or $pq$. The possibilities ${\rm dim}\overline{H}= p,
pq$ or $p^2q$ lead, respectively to the contradictions $p^2q^2=
p+cq^2$, $p^2q^2=pq+cq^2$ and $p^2q^2=p^2q+cq^2$. Hence these are
also discarded, and therefore ${\rm dim}\overline{H}=pq^2$. This
implies that $H$ is upper semisolvable, by Lemma \ref{lem6}.

Finally, the theorem follows from Lemma \ref{lem5}, \ref{lem6} and
\ref{lem7}.
\end{proof}

As an immediate consequence of Theorem \ref{thm1}, we have the
following corollary.
\begin{corollary}\label{cor1}
If $p^2<q$ and $H$ is simple as a Hopf algebra then $H$ is
isomorphic to a Radford's biproduct $R\# kG$, where $kG$ is the
group algebra of group $G$ of order $p^2$, $R$ is a semisimple
Yetter-Drinfeld Hopf algebra in ${}^{kG}_{kG}\mathcal{YD}$ of
dimension $q^2$.
\end{corollary}

In fact, examples of nontrivial semisimple Hopf algebras of
dimension $p^2q^2$ which are Radford's biproducts in such a way, and
are simple as Hopf algebras do exists. A construction of such
examples as twisting deformations of certain groups appears in
\cite[Remark 4.6]{Galindo}.

\section{Applications}\label{sec4}
\subsection{Semisimple Hopf algebras of dimension $9q^2$}
In this subsection, we shall prove the following theorem.
\begin{theorem}\label{thm2}
If $H$ is a semisimple Hopf algebra of dimension $9q^2$ then $H$ is
either semisolvable or isomorphic to a Radford's biproduct $R\# kG$,
where $kG$ is the group algebra of group $G$ of order $9$, $R$ is a
semisimple Yetter-Drinfeld Hopf algebra in
${}^{kG}_{kG}\mathcal{YD}$ of dimension $q^2$.
\end{theorem}

By Theorem \ref{thm1}, it suffices to consider the case $q=5$ and
$7$.
\begin{lemma}\label{lem8}
If $q=5$ then $H$ is either semisolvable or isomorphic to a
Radford's biproduct $R\# kG$, where $kG$ is the group algebra of
group $G$ of order $9$, $R$ is a semisimple Yetter-Drinfeld Hopf
algebra in ${}^{kG}_{kG}\mathcal{YD}$ of dimension $25$.
\end{lemma}
\begin{proof} By Lemma \ref{lem1}, \ref{lem2} and \ref{lem3}, if ${\rm
dim}H=3^2\times 5^2$ then $H$ is of one of the following types as an
algebra:
$$(1,25;5,8), (1,75;5,6), (1,3;3,8;5,6), (1,9;3,6;9,2), (1,9;3,24),
(1,45;3,20).$$ If $H$ is of type $(1,25;5,8)$ as an algebra then
Lemma \ref{lem7} shows that $H$ is upper semisolvable. If $H$ is of
type $(1,75;5,6)$ as an algebra then Lemma \ref{lem6} shows that $H$
is upper semisolvable. If $H$ is of type $(1,3;3,8;5,6)$ as an
algebra then Corollary \ref{lem4} shows that $H$ has a quotient Hopf
algebra of dimension $75$. Hence, Lemma \ref{lem6} shows that $H$ is
upper semisolvable. The lemma then follows from Lemma \ref{lem5}.
\end{proof}

\begin{remark}
The computation in the proof of Lemma \ref{lem8} is partly handled
by a computer. For example, it is easy to write a computer program
by which one finds out all non-negative integers $n_1,n_2,n_3,n_4$
such that $225=n_1+9n_2+81n_3+25n_4$, and then one can eliminate
those which can not be algebra types of $H$ by using Lemma
\ref{lem1}, \ref{lem2} and \ref{lem3}. The computations in the
followings are handled similarly.
\end{remark}

\begin{lemma}\label{lem9}
If $q=7$ then $H$ is either semisolvable or isomorphic to a
Radford's biproduct $R\# kG$, where $kG$ is the group algebra of
group $G$ of order $9$, $R$ is a semisimple Yetter-Drinfeld Hopf
algebra in ${}^{kG}_{kG}\mathcal{YD}$ of dimension $49$.
\end{lemma}
\begin{proof} By Lemma \ref{lem1}, \ref{lem2} and \ref{lem3}, if ${\rm
dim}H=3^2\times 7^2$ then $H$ is of one of the following types as an
algebra:
$$(1,3;3,14;5,6;9,2), (1,3;3,32;5,6), (1,3;3,16;7,6),(1,21;3,14;7,6),$$
$$(1,49;7,8), (1,147;7,6), (1,9;3,12;9,4), (1,9;3,30;9,2), (1,9;3,48), (1,63;3,42).$$
Corollary \ref{lem4} shows that $H$ can not be of type
$(1,3;3,14;5,6;9,2), (1,3;3,32;5,6)$ as an algebra, since it
contradicts Nichols-Zoeller Theorem. The lemma then follows from a
similar argument as in Lemma \ref{lem8}.
\end{proof}
\begin{corollary}\label{cor2}
If $H$ is a semisimple Hopf algebra of dimension $9q^2$ and is
simple as a Hopf algebra then $H$ is isomorphic to a Radford's
biproduct $R\# kG$, where $kG$ is the group algebra of group $G$ of
order $9$, $R$ is a semisimple Yetter-Drinfeld Hopf algebra in
${}^{kG}_{kG}\mathcal{YD}$ of dimension $q^2$.
\end{corollary}

\subsection{Semisimple Hopf algebras of dimension $25q^2$}
In this subsection, we shall prove the following theorem.
\begin{theorem}\label{thm3}
If $H$ is a semisimple Hopf algebra of dimension $25q^2$ then $H$ is
either semisolvable or isomorphic to a Radford's biproduct $R\# kG$,
where $kG$ is the group algebra of group $G$ of order $25$, $R$ is a
semisimple Yetter-Drinfeld Hopf algebra in
${}^{kG}_{kG}\mathcal{YD}$ of dimension $q^2$.
\end{theorem}

By Theorem \ref{thm1}, it suffices to consider the case $7\leq q\leq
23$.
\begin{lemma}\label{lem10}
If $q=7$ then $H$ is either semisolvable or isomorphic to a
Radford's biproduct $R\# kG$, where $kG$ is the group algebra of
group $G$ of order $25$, $R$ is a semisimple Yetter-Drinfeld Hopf
algebra in ${}^{kG}_{kG}\mathcal{YD}$ of dimension $49$.
\end{lemma}
\begin{proof} By Lemma \ref{lem1}, \ref{lem2} and \ref{lem4-1}, if ${\rm
dim}H=5^2\times 7^2$ then $H$ is of one of the following types as an
algebra:
$$(1,35;5,28;7,10),(1,49;7,24), (1,245;7,20), (1,175;5,42),
(1,25;5,48).$$ We shall prove that $H$ can not be of type
$(1,35;5,28;7,10)$ as an algebra. The lemma then will follow from
Lemma \ref{lem5}, \ref{lem6} and \ref{lem7}.

Suppose on the contrary that $H$ is of type $(1,35;5,28;7,10)$ as an
algebra. The group $G(H^*)$ acts by left multiplication on the set
$X_5$. The set $X_5$ is a union of orbits which have length $1,5$ or
$7$. By Lemma \ref{lem1} (1), $G[\chi]$ is a proper subgroup of
$G(H^*)$ for every $\chi\in X_5$. Hence, there does not exist orbits
with length $1$. Accordingly, every orbit has length $7$ and the
order of $G[\chi]$ is $5$ for every $\chi\in X_5$. In particular,
the decomposition of $\chi\chi^*$ does not contain irreducible
characters of degree $7$.

Let $\chi',\chi$ be distinct irreducible characters of degree $5$.
Suppose that there exists $\chi_7\in X_7$ such that
$m(\chi_7,\chi'\chi^*)>0$. Then there must exist $\varepsilon\neq
g\in G(H^*)$ such that $m(g,\chi'\chi^*)=1$. From this observation,
we know $\chi'=g\chi$ and $\chi'\chi^*=g\chi\chi^*$. Since
$\chi\chi^*$ does not contain irreducible characters of degree $7$,
$\chi'\chi^*$ does not contain such characters, too. This
contradicts the assumption. Therefore, $\chi'\chi^*$ is a sum of
irreducible characters of degree $1$ or $5$. It follows that
$G(H^*)\cup X_5$ spans a standard subalgebra of $R(H)$, and $H$ has
a quotient Hopf algebra of dimension $735$. This contradicts the
Nichols-Zoeller Theorem \cite{Nichols2}.
\end{proof}

\begin{lemma}\label{lem16}
Let $H$ be a semisimple Hopf algebra of dimension $25q^2$, where
$q=11,17,19$. If $|G(H^*)|=5$ or $5q$ then $H$ has a quotient Hopf
algebra of dimension $|G(H^*)|+25a$, where $a$ is the cardinal
number of $X_5$.
\end{lemma}

\begin{proof}In fact, it can be checked directly that $G[\chi]=5$ for every
$\chi\in X_5$. Then the lemma follows from a similar argument as in
the proof of Lemma \ref{lem10}.
\end{proof}

\begin{lemma}\label{lem11}
If $q=11$ then $H$ is either semisolvable or isomorphic to a
Radford's biproduct $R\# kG$, where $kG$ is the group algebra of
group $G$ of order $25$, $R$ is a semisimple Yetter-Drinfeld Hopf
algebra in ${}^{kG}_{kG}\mathcal{YD}$ of dimension $121$.
\end{lemma}
\begin{proof} By Lemma \ref{lem1}, \ref{lem2} and \ref{lem4-1}, if ${\rm
dim}H=5^2\times 11^2$ then $H$ is of one of the following types as
an algebra:
$$(1,5;5,24;11,20), (1,55;5,22;11,20), (1,121;11,24),$$$$ (1,605;11,20),
(1,275;5,110),(1,25;5,20;25,4), (1,25;5,70;25,2),(1,25;5,120).$$ By
Lemma \ref{lem16}, if $H$ is of type $(1,5;5,24;11,20)$ or
$(1,55;5,22;11,20)$ as an algebra then $H$ has a quotient Hopf
algebra of dimension $605$. Then $H$ is upper semisolvable by Lemma
\ref{lem6}. The lemma then follows from Lemma \ref{lem5}, \ref{lem6}
and \ref{lem7}.
\end{proof}

\begin{lemma}\label{lem12}
If $q=13$ then $H$ is either semisolvable or isomorphic to a
Radford's biproduct $R\# kG$, where $kG$ is the group algebra of
group $G$ of order $25$, $R$ is a semisimple Yetter-Drinfeld Hopf
algebra in ${}^{kG}_{kG}\mathcal{YD}$ of dimension $169$.
\end{lemma}
\begin{proof} By Lemma \ref{lem1}, \ref{lem2} and \ref{lem4-1}, if ${\rm
dim}H=5^2\times 13^2$ then $H$ is of one of the following types as
an algebra:
$$(1,169;13,24),(1,845;13,20),(1,325;5,156),$$
$$(1,25;5,18;25,6),(1,25;5,168),(1,25;5,68;25,4),(1,25;5,118;25,2).$$
The lemma then follows directly from Lemma \ref{lem5}, \ref{lem6}
and \ref{lem7}.
\end{proof}

\begin{lemma}\label{lem13}
If $q=17$ then $H$ is either semisolvable or isomorphic to a
Radford's biproduct $R\# kG$, where $kG$ is the group algebra of
group $G$ of order $25$, $R$ is a semisimple Yetter-Drinfeld Hopf
algebra in ${}^{kG}_{kG}\mathcal{YD}$ of dimension $289$.
\end{lemma}
\begin{proof} By Lemma \ref{lem1}, \ref{lem2} and \ref{lem4-1}, if ${\rm
dim}H=5^2\times 17^2$ then $H$ is of one of the following types as
an algebra:
$$(1,85;5,170;17,10),$$
$$ (1,289;17,24), (1,1445;17,20), (1,425;5,272),(1,25;5,288),(1,25;5,238;25,2)$$
$$(1,25;5,38;25,10), (1,25;5,88;25,8), (1,25;5,138;25,6), (1,25;5,188;25,4).$$
By Lemma \ref{lem16}, if $H$ is of type $(1,85;5,170;17,10)$ as an
algebra then $H$ has a quotient Hopf algebra of dimension $4335$.
This contradicts the Nichols-Zoeller Theorem \cite{Nichols2}. The
lemma then follows from Lemma \ref{lem5}, \ref{lem6} and \ref{lem7}.
\end{proof}

\begin{lemma}\label{lem14}
If $q=19$ then $H$ is either semisolvable or isomorphic to a
Radford's biproduct $R\# kG$, where $kG$ is the group algebra of
group $G$ of order $25$, $R$ is a semisimple Yetter-Drinfeld Hopf
algebra in ${}^{kG}_{kG}\mathcal{YD}$ of dimension $361$.
\end{lemma}
\begin{proof} By Lemma \ref{lem1}, \ref{lem2} and \ref{lem4-1}, if ${\rm
dim}H=5^2\times 19^2$ then $H$ is of one of the following types as
an algebra:
$$(1,5;5,22;19,20;25,2), (1,5;5,72;19,20),$$$$ (1,361;19,24), (1,1805;19,20),(1,475;5,342),$$
$$(1,25;5,10;25,14), (1,25;5,60;25,12), (1,25;5,110;25,10),(1,25;5,360),$$
$$(1,25;5,160;25,8),(1,25;5,210;25,6), (1,25;5,260;25,4), (1,25;5,310;25,2).$$
By Lemma \ref{lem16}, if $H$ is of type $(1,5;5,22;19,20;25,2)$ as
an algebra then $H$ has a quotient Hopf algebra of dimension $555$.
This contradicts the Nichols-Zoeller Theorem \cite{Nichols2}. Again
by Lemma \ref{lem16}, if $H$ is of type $(1,5;5,72;19,20)$ as an
algebra then $H$ has a quotient Hopf algebra of dimension $1805$.
Then $H$ is upper semisolvable by Lemma \ref{lem6}. The lemma then
follows from Lemma \ref{lem5}, \ref{lem6} and \ref{lem7}.
\end{proof}

\begin{lemma}\label{lem15}
If $q=23$ then $H$ is either semisolvable or isomorphic to a
Radford's biproduct $R\# kG$, where $kG$ is the group algebra of
group $G$ of order $25$, $R$ is a semisimple Yetter-Drinfeld Hopf
algebra in ${}^{kG}_{kG}\mathcal{YD}$ of dimension $529$.
\end{lemma}
\begin{proof} By Lemma \ref{lem1}, \ref{lem2} and \ref{lem4-1}, if ${\rm
dim}H=5^2\times 23^2$ then $H$ is of one of the following types as
an algebra:
$$(1,529;23,24), (1,2645;23,20),(1,575;5,506),$$$$(1,25;5,28;25,20),(1,25;5,528),(1,25;5,478;25,2)$$
$$(1,25;5,278;25,10), (1,25;5,328;25,8), (1,25;5,378;25,6), (1,25;5,428;25,4),$$
$$(1,25;5,78;25,18), (1,25;5,128;25,16), (1,25;5,178;25,14), (1,25;5,228;25,12).$$
The lemma then follows directly from Lemma \ref{lem5}, \ref{lem6}
and \ref{lem7}.
\end{proof}

\begin{corollary}\label{cor3}
If $H$ is a semisimple Hopf algebra of dimension $25q^2$ and is
simple as a Hopf algebra then $H$ is isomorphic to a Radford's
biproduct $R\# kG$, where $kG$ is the group algebra of group $G$ of
order $25$, $R$ is a semisimple Yetter-Drinfeld Hopf algebra in
${}^{kG}_{kG}\mathcal{YD}$ of dimension $q^2$.
\end{corollary}

\end{document}